\documentclass[reqno]{amsart}
\usepackage{amssymb, amsmath, amsthm, amscd, mathrsfs}

\usepackage{paralist}
\usepackage{cite}
\usepackage{bigints}
\usepackage{dsfont}
\usepackage{empheq}
\usepackage{amssymb}
\usepackage{cases}
\usepackage{enumitem}
\usepackage{color}
\usepackage{cancel}
\allowdisplaybreaks

\usepackage{algorithm}

\usepackage{hyperref}
\numberwithin{equation}{section}

\theoremstyle{plain}
\newtheorem{theorem}{Theorem}[section]
\newtheorem{corollary}{Corollary}
\newtheorem{lemma}[theorem]{Lemma}
\newtheorem{proposition}{Proposition}
\theoremstyle{definition}
\newtheorem{definition}[theorem]{Definition}
\newtheorem{remark}{Remark}

\def \d {\mathrm{d}}

\title[The biharmonic heat with dynamic boundary conditions
] 
{
The Biharmonic Heat Equation with General Dynamic Boundary Conditions}

\author{Salah-Eddine Chorfi}
\author{Fouad Et-tahri}
\author{Lahcen Maniar}

\address{S. E. Chorfi, L. Maniar, Faculty of Sciences Semlalia, Cadi Ayyad University, Laboratory of Mathematics, Modeling and Automatic Systems, B.P. 2390, Marrakesh, Morocco}
\address{L. Maniar, The UM6P-Vanguard Center, University Mohammed VI Polytechnic, Rabat Campus, Morocco}
\address{F. Et-tahri, Faculty of Sciences-Agadir, Lab-SIV, Ibn Zohr University, B.P. 8106, Agadir, Morocco}

\email{s.chorfi@uca.ac.ma, maniar@uca.ac.ma}
\email{fouad.et-tahri@edu.uiz.ac.ma}

\makeatletter 
\@namedef{subjclassname@2020}{%
	\textup{2020} Mathematics Subject Classification}
\makeatother


\begin{document}
	\begin{abstract}
		In this work, we initiate the study of the biharmonic heat equation in a spatial bounded domain subject to dynamic boundary conditions involving the bi-Laplace-Beltrami operator on the boundary. The boundary heat equation is coupled to the interior one via a normal derivative term. By combining the sesquilinear form method and semigroup theory, we establish substantial qualitative properties of the fourth-order parabolic equation; in particular, the self-adjointness of the associated operator, compactness of its resolvent, and further spectral properties. We also investigate the generation of a $C_0$-semigroup and analyze its main properties: analyticity, compactness, eventual positivity, and eventual $L^\infty$-contractivity.
	\end{abstract}
	\maketitle

	\section{Introduction and statement of the problem} \label{sec1}
    Fourth-order parabolic equations arise naturally in the modeling of various physical and applied phenomena. They appear, for instance, in the study of flexible structures, nonlinear elasticity, epitaxial thin-film growth, and image segmentation, etc; see, e.g., \cite{ES86,KSW03,GBS06} and the references therein.
    
	We study the biharmonic heat equation with dynamic boundary conditions focusing on its well-posedness and $L^2$-regularity of the solutions, as well as other qualitative properties. More precisely, let $T>0$ and let $\Omega \subset \mathbb{R}^N$ ($N\geq 2$) be an open bounded connected subset with boundary $\Gamma=\partial\Omega$ of class $C^4$. We denote
	$$\Omega_T =(0,T)\times \Omega \qquad \text{and} \qquad\Gamma_T =(0,T)\times \Gamma,$$
	and we consider the biharmonic heat equation with dynamic boundary conditions given by
	\begin{empheq}[left = \empheqlbrace]{alignat=2}
		\begin{aligned}
			&\partial_t y + d\Delta^2 y = f(t,x), &&\quad\text{in } \Omega_T , \\
			&\partial_t y_{\Gamma} + \delta\Delta^2_{\Gamma} y_{\Gamma} - d\partial_{\nu} (\Delta y)= f_\Gamma(t,x), &&\quad\text{on } \Gamma_T, \\
			&\kappa\partial_{\nu}(\Delta y)=y- y_{\Gamma},   &&\quad\text{on } \Gamma_T, \\
			&\partial_{\nu}y=0, &&\quad\text{on } \Gamma_T, \\
			&(y(0,\cdot),y_{\Gamma}(0,\cdot))=(y_0, y_{0,\Gamma}),   &&\quad\text{in }\Omega\times\Gamma, \label{eq1to4}
		\end{aligned}
	\end{empheq}
	where the initial states are denoted by $(y_0, y_{0,\Gamma})\in L^2(\Omega)\times L^2(\Gamma)$, while the source terms are $f\in L^2(\Omega_T)$ and $f_\Gamma\in L^2(\Gamma_T)$. The constants $d> 0$, $\delta>0$ denote the diffusivities, and $\kappa>0$ is a given constant. We denote by $\partial_{\nu} y :=(\nabla y\cdot \nu) _{|\Gamma}$  the normal derivative of $y$, where $\nu$ represents the outer unit normal field on $\Gamma$. The notation $\Delta^2$ stands for the bi-Laplacian operator with respect to the space variable in $\Omega$. Similarly, $\Delta_\Gamma^2$ stands for the bi-Laplace-Beltrami operator, where $\Delta_\Gamma$ is the Laplace-Beltrami operator on $\Gamma,$ regarded as a compact Riemannian submanifold, (see Section \ref{sec2} for more details).
	
	We emphasize that equation \eqref{eq1to4} is a fourth-order parabolic equation in both $\Omega$ and $\Gamma$. Therefore, we naturally impose a second boundary condition that we chose as a Neumann boundary condition $\partial_{\nu} y=0$ on $\Gamma$ for simplicity. Note that we could instead consider a Dirichlet boundary condition $\Delta y=0$ on $\Gamma$, or even mixed boundary conditions $\Delta y =0$ on $\Gamma_1$ and $\partial_\nu y=0$ on $\Gamma_2$, where $\Gamma_1$ and $\Gamma_2$ are open subsets of $\Gamma$ such that $\Gamma=\overline{\Gamma_1} \cup \overline{\Gamma_2} \text{ and } \Gamma_1 \cap \Gamma_2=\varnothing$. Notice that the boundary condition $\eqref{eq1to4}_2$ is dynamic since it involves the time derivative of the state on the boundary. The condition $\eqref{eq1to4}_3$ is a Robin-type coupling, which reduces to the side condition $y_{\Gamma}=y\rvert_\Gamma$ as $\kappa\to 0$. Considering our general methodology, it should also be pointed out that our results hold for general elliptic fourth-order operators with lower-order terms.
	
    The literature has given considerable attention to fourth-order parabolic equations. Let us start by briefly reviewing the case of fourth-order equations with classical boundary conditions.
	\subsection*{Static boundary conditions}
	Fourth-order parabolic equations exhibit some remarkable properties that are different from their second-order counterparts, such as lack of maximum principle and lack of positivity preserving \cite{Be'87, FGG'08, GG'08}. In this context, it is worth mentioning the recent paper \cite{GMO'21} that studies the positivity of solutions to the linear and semilinear biharmonic heat equations in $\mathbb R^N$ (in which most of the works have been carried out). Regarding higher-order linear parabolic equations, we refer to the paper \cite{BG'13}, while we refer to the monograph \cite{GGS'10} for nonlinear elliptic equations in bounded domains.
	
	\subsection*{Dynamic boundary conditions}
Dynamic boundary conditions appear in many branches of applied mathematics. They model phenomena of bulk–surface (interface) interactions in heat transfer and fluid mechanics \cite{ACM'1, ACM'2, La'32, MMS'17}, in the biology of cells \cite{ACM'3, EFPT'18}, population dynamics \cite{FH'11}, models of reaction-diffusion \cite{ACMO'4, EC'25}, etc. Second-order parabolic equations with dynamic boundary conditions have been extensively studied in recent decades. To better understand the physical derivation of such boundary conditions, we refer to \cite{Go'06}. For a more recent derivation via the Carslaw-Jaeger constitutive law, we refer the reader to the paper \cite{Sa'20}.
	
	Nevertheless, only a few works investigate fourth-order equations subject to boundary conditions of dynamic type. We are only aware of the following papers in this context: in the smooth setting of a bounded domain with a $C^4$ boundary, the paper \cite{FGGR'08} gives first results for the bi-Laplacian with general Wentzell boundary conditions such as essential self-adjointness and boundedness from below. Recently, in \cite{GGGR'20} the authors have considered the fourth-order Wentzell heat equation in a uniformly regular domain of class $C^{2+\varepsilon}$ whose boundary is also a uniformly regular domain of class $C^{2+\varepsilon}$. In the less regular setting of the Lipschitz boundary, the recent paper \cite{DKP'21} investigates the main properties of the operator such as the characterization of the domain, the generation of the semigroup, its analyticity, asymptotic behavior, and eventual positivity. Note that this kind of equation is closely related to linear versions of the Cahn–Hilliard equation with dynamic boundary conditions that have been studied in \cite{Gal'07, MZ'05, RZ'03}.
	
	Unlike the works mentioned above, our system \eqref{eq1to4} incorporates a surface diffusion term given by the bi-Laplace-Beltrami operator, which has not been considered before in the literature. This allows us to extend some previous models for the biharmonic heat equation in $\Omega$ to cover bulk-surface interactions by including their boundary counterparts in $\Gamma$.
	
	The remainder of the work is organized as follows. In Section \ref{sec2}, we recall some preliminaries that will be useful for easily understanding the subsequent sections. In Section \ref{sec3}, we prove the well-posedness of system \eqref{eq1to4} and the regularity of its solution. We will use the method of sesquilinear forms to prove the main properties of the associated operator and the generated semigroup such as self-adjointness, analyticity, compactness, asymptotic behavior, etc. Finally, in Section \ref{sec4}, we study the eventual positivity and eventual $L^\infty$-contractivity of the semigroup.
	
	\section{Preliminaries} \label{sec2}
	We start by gathering some basic results that will be used throughout the sequel.
		\subsection*{Notation}
        We denote by $D(\Omega)$ the space of smooth functions with compact support in $\Omega$, i.e., $D(\Omega) = \mathcal{C}_c^\infty(\Omega)$.
        
		For all $k\in\mathbb{N}$, $H^{k}(\Omega)$ and $H^{k}(\Gamma)$ will denote the classical Sobolev spaces based on $L^{2}(\Omega)$ and $L^{2}(\Gamma)$, respectively. For any $u\in C^{k}(\overline{\Omega})$ and any $j=0,\cdots, k-1$ we define the traces
		$$\gamma_{j}u=\frac{\partial^{j}u}{\partial\nu^{j}}.$$
		These linear operators are extended continuously to the
		larger space $H^{k}(\Omega)$, see, e.g., \cite[Theorem 8.3]{LM'72}. We set 
		$$H^{k-j-1/2}(\Gamma):=\gamma_{j}(H^{k}(\Omega)),\qquad j=0,\cdots, k-1.$$
		In the sequel, we will simply write $u$  and $\partial_\nu u$ instead of $\gamma_{0}u$ and $\gamma_{1}u$, respectively.
        
	\subsection{Geometric setup and integration by parts}
	We will often use the so-called Green's formula in $\Omega$:
	$$ \int_{\Omega} \Delta u\, v \,\d x=-\int_{\Omega} \nabla u \cdot \nabla v \,\d x + \int_{\Gamma} \partial_{\nu} u\, v \,\d S, \qquad u \in H^{2}(\Omega), \; v \in H^{1}(\Omega).$$
	The boundary $\Gamma$ is a compact Riemannian manifold of dimension $(N-1)$ without boundary. We refer to \cite{Ta'11} for more details. Let $g$ be the Riemannian metric on $\Gamma$ induced by the natural embedding $\Gamma \hookrightarrow \mathbb{R}^N$. We fix a coordinate system $x=(x^j)$ and denote by $\displaystyle \left(\frac{\partial}{\partial x^j}\right)$ the corresponding tangent vector field. In local coordinates, $g$ is given by $g_{ij}:=\left\langle \dfrac{\partial }{\partial x^i}, \dfrac{\partial }{\partial x^j}\right\rangle$. We define the tangential gradient locally for any smooth function $y$ on $\Gamma$ by
	$$\nabla_\Gamma y :=\sum_{i,j=1}^{N-1} g^{ij} \frac{\partial y}{\partial x^j} \frac{\partial }{\partial x^i},$$
	where we denote $g=(g_{ij})$, $(g^{ij})$ its inverse and $|g|=\det(g_{ij})$. It is well known that $\nabla_\Gamma y$ is the projection of the standard Euclidean gradient $\nabla y$ onto the tangent space on $\Gamma$, that is,
	\begin{equation}
		\nabla_\Gamma y = \nabla y -(\partial_{\nu} y) \nu. \label{eqtgrad}
	\end{equation}
	The Laplace-Beltrami operator on $\Gamma$ with respect to the induced
	metric $g$ is given by
	\begin{equation}\label{eqlb}
		\Delta_\Gamma =\frac{1}{\sqrt{|g|}} \sum_{i,j=1}^{N-1} \frac{\partial}{\partial x^i} \left(\sqrt{|g|}\, g^{ij} \frac{\partial}{\partial x^j}\right).
	\end{equation}
	Since $\Gamma$ is a compact Riemannian manifold without boundary, the following divergence formula holds:
	\begin{equation}\label{sdt}
		\int_\Gamma \Delta_\Gamma y \, z \,\d S =- \int_\Gamma \langle \nabla_\Gamma y, \nabla_\Gamma z \rangle_\Gamma \,\d S, \qquad y\in H^2(\Gamma),\; z\in H^1(\Gamma),
	\end{equation}
	which is analog to Green's formula. 
	We can see that formula \eqref{sdt} makes sense for 
		$y, z\in H^1(\Gamma)$, where the integral in the left-hand side has to be interpreted in the distributional sense, as $\Delta_{\Gamma}y\in H^{-1}(\Gamma)$.
		Clearly $\Delta_{\Gamma}$ is not injective, but by \eqref{sdt} we have
		$$ \int_\Gamma (-\Delta_\Gamma y+y) \, y \,\d S =\|y\|^{2}_{H^{1}(\Gamma)}.$$
		Consequently, the operator $-\Delta_\Gamma +1$ defined by  $(-\Delta_\Gamma +1)y=-\Delta_\Gamma y+y$ is a topological and algebraic isomorphism from $H^1(\Gamma)$ into $H^{-1}(\Gamma)$. Moreover, by the elliptic regularity (see \cite[p.~361]{Ta'11}), $(-\Delta_\Gamma +1)^{-1}$ is bounded from $H^{k-1}(\Gamma)$ into $H^{k+1}(\Gamma)$ ($k\in \mathbb{N}\cup \{0\}$). Therefore, 
		\begin{equation}\label{iso}
			-\Delta_\Gamma +1: H^{k+1}(\Gamma) \rightarrow H^{k-1}(\Gamma)\quad \text{is an isomorphism}.
	      \end{equation}
    
	\begin{remark}
		In the one-dimensional case $N=1$, $\Gamma$ is a manifold of dimension $0$. Consequently, the tangential operators are trivial, i.e., $\nabla_{\Gamma}=0$ and $\Delta_{\Gamma}=0$, which motivates the assumption $N\ge 2$.
	\end{remark}

	\subsection{Functional setting}
	We introduce the functional setting for our system. We denote the $N$-dimensional Lebesgue measure on $\Omega$ by $\d x$ (or $|\cdot|$) and the $(N-1)$-dimensional surface measure on $\Gamma$ by $\d S$ (or simply $|\cdot|_\Gamma$).
	Then we consider the real Hilbert space
	$$\mathbb{L}^2:=L^2(\Omega, \d x)\times L^2(\Gamma, \d S)$$
	that represents the state space, endowed with the inner product given by
	$$\langle (y,y_\Gamma),(z,z_\Gamma)\rangle_{\mathbb{L}^2} =\langle y,z\rangle_{L^2(\Omega)} +\langle y_\Gamma,z_\Gamma\rangle_{L^2(\Gamma)}.$$
	We set $\mathds{1}:=(1,1)$ and $\mathbf{0}:=(0,0)$, which are constant functions in $\mathbb{L}^2$.
	
		We also define the space
		$$\mathbb{L}^\infty:=L^\infty(\Omega, \d x)\times L^\infty(\Gamma, \d S).$$
	
		The space $\mathbb{L}^2$ can be identified with the space $L^2\left(\overline{\Omega}, \d \mu\right)$, where the measure $\mu$ on $\overline{\Omega}$ is defined for every measurable subset $E\subset \overline{\Omega}$ by
		$$\mu(E):=|\Omega\cap E|+|\Gamma\cap E|_\Gamma.$$
		The space $\mathbb{L}^\infty$ can also be identified with the space $L^\infty\left(\overline{\Omega}, \d \mu\right)$, since $\mu$ is a Radon measure on the Borel $\sigma$-algebra $\mathscr{B}\left(\overline{\Omega}\right)$.
	
	Next, we introduce the spaces 
	$$\mathcal{H}^m:=H^m(\Omega)\times H^m(\Gamma), \quad m\in\mathbb{N},$$ which are Hilbert spaces endowed with the standard inner products
	\begin{align*}\left\langle\left(y, y_{\Gamma}\right),\left(z, z_{\Gamma}\right)\right\rangle_{\mathcal{H}^{m}}&=\left\langle y, z\right\rangle_{H^{m}(\Omega)}+\left\langle  y_{\Gamma},z_{\Gamma}\right\rangle_{H^{m}(\Gamma)}.
	\end{align*}
	
	\subsection{Miscellaneous facts on eventual positivity}
	We briefly recall some useful facts that will be used later. We refer, e.g., to \cite{GM'20} and \cite{DG18} for relevant results.
	
	Let $X$ be a finite measure space. Then the space $L^2(X)$ is a Hilbert lattice. Let $(\mathcal{T}(t))_{t \geq 0}$ be a real $C_0$-semigroup on $L^2(X)$ whose generator $A$ is self-adjoint with compact resolvent.
	
	\begin{definition} We adopt the following notions: 
		\begin{itemize}
			\item[$\bullet$] The semigroup $(\mathcal{T}(t))_{t \geq 0}$ is called \textit{uniformly eventually positive} if
			\begin{equation}\label{ep}
				\exists t_0 \ge 0,\; \forall f \in L^2(X),\quad f \geq 0 \Longrightarrow \mathcal{T}(t) f \ge 0 \quad \forall t \geq t_0.
			\end{equation}
			\item[$\bullet$] The semigroup $(\mathcal{T}(t))_{t \geq 0}$ is called \textit{positive} if \eqref{ep} holds for $t_0=0$.
			\item[$\bullet$] The semigroup $(\mathcal{T}(t))_{t \geq 0}$ is called \textit{uniformly eventually} $L^\infty$-\textit{contractive} if
			\begin{equation}\label{elc}
				\exists t_0 \ge 0,\; \forall f \in L^2(X),\; f \in L^\infty(X) \Longrightarrow \|\mathcal{T}(t) f\|_{L^\infty(X)} \le \|f\|_{L^\infty(X)} \quad \forall t \geq t_0.
			\end{equation}
			\item[$\bullet$] The semigroup $(\mathcal{T}(t))_{t \geq 0}$ is called $L^\infty$-\textit{contractive} if \eqref{elc} holds for $t_0=0$.
			\item[$\bullet$] The semigroup $(\mathcal{T}(t))_{t \geq 0}$ is called \textit{uniformly eventually sub-Markovian} if it is uniformly eventually positive and uniformly eventually $L^\infty$-contractive.
			\item [$\bullet$] The semigroup $(\mathcal{T}(t))_{t \geq 0}$ is called \textit{uniformly eventually Markovian} if it is uniformly eventually sub-Markovian and $\mathcal{T}(t)1=1$.
		\end{itemize}
	\end{definition}
	
	For any linear operator $B$, we denote by
	$$s(B):=\sup\{\mathrm{Re}\,\lambda \;\colon \lambda\in \sigma(B)\}$$
	its \textit{spectral bound}, and we set
	$$D(B^\infty):=\bigcap_{n \in \mathbb{N}} D\left(B^n\right).$$
	\begin{proposition}[see \cite{GM'20}]\label{subM}
		We assume that $s(-A)=0$ and the associated eigenspace $E_0$ be a one-dimensional space spanned by $1$. We further assume the continuous embedding $D(A^\infty)\hookrightarrow L^{\infty}(X)$. Then, the semigroup $(\mathcal{T}(t))_{t \geq 0}$ is uniformly eventually sub-Markovian.
	\end{proposition}
	
	\section{Well-posedness and regularity of the solution} \label{sec3}
	In this section, we study the well-posedness and regularity of the system \eqref{eq1to4}. 
	This system can be written as follows
\begin{equation*}	
	\text{(ACP)}\quad \begin{cases}
		\partial_t Y=\mathcal{A} Y+ F, \quad 0<t<T, \nonumber\\
		Y(0)=Y_0, \nonumber
	\end{cases}
	\end{equation*}
	where $Y:=(y,y_{\Gamma})$, $F:=(f,f_\Gamma)$, $Y_0:=(y_0, y_{0,\Gamma})$ and the linear operator $$\mathcal{A} \colon D(\mathcal{A}) \subset \mathbb{L}^2 \longrightarrow \mathbb{L}^2$$
	is given by
	\begin{eqnarray}
		&&\mathcal{A}=\begin{pmatrix} -d\Delta^2\quad & 0\\ d\partial_\nu \Delta\quad & -\delta\Delta^2_\Gamma\end{pmatrix}
	\end{eqnarray} 
	with domain
	\begin{align*}
		&&D(\mathcal{A}):=
		\{(y,y_{\Gamma})\in \mathcal{H}^{4}\colon \partial_{\nu}y=0\quad\mbox{and}\quad \kappa\partial_{\nu}(\Delta y)=y-  y_{\Gamma}\}.
	\end{align*}
	In the following part, we mainly borrow our terminology from \cite{Ou'05}. Now we introduce the bilinear form given by
	\begin{align*}
		\mathfrak{a}[(y,y_\Gamma),(z,z_\Gamma)]&= d\int_\Omega \Delta y \Delta z \,\d x 
		+ \delta\int_\Gamma \Delta_\Gamma y_\Gamma \Delta_\Gamma z_\Gamma \,\d S\\
		& \quad +\frac{d}{\kappa}\int_{\Gamma}( y-  y_{\Gamma})( z-  z_{\Gamma})\d S,
	\end{align*}
	with domain 
	$$D(\mathfrak{a})=
		\{(y,y_{\Gamma})\in \mathcal{H}^{2}\colon \partial_{\nu}y=0 \},
	$$ 
	in the Hilbert space $\mathbb{L}^2$. The operator $\widetilde{\mathcal{A}}$ associated to the form $\mathfrak{a}$ is defined as follows
	\begin{equation}
	    \begin{aligned}
		D(\widetilde{\mathcal{A}})&:=\left\{(y,y_\Gamma) \in D(\mathfrak{a}), \text{ there exists } (f,f_\Gamma) \in \mathbb{L}^2 \text{ such that } \right.\\
		&\quad \left. \mathfrak{a}[(y,y_\Gamma),(z,z_\Gamma)]=\langle (f,f_\Gamma), (z,z_\Gamma)\rangle_{\mathbb{L}^2} \text{ for all } (z,z_\Gamma) \in D(\mathfrak{a}) \right\}, \\
		\widetilde{\mathcal{A}}(y,y_\Gamma)&:=(f,f_\Gamma) \qquad \text{ for all }(y,y_\Gamma) \in D(\widetilde{A}).
	\end{aligned}
	\end{equation}
    Next, we prove the main properties of the form $\mathfrak{a}.$
	\begin{lemma} \label{genform}
	The form $\mathfrak{a}$ is densely defined, symmetric, accretive, and closed.
	\end{lemma}
	\begin{proof}
	The form $\mathfrak{a}$ is symmetric and accretive, that is, for all $(y,y_\Gamma), (z,z_\Gamma)\in D(\mathfrak{a})$,
	$$\mathfrak{a}[(y,y_\Gamma),(z,z_\Gamma)]=\mathfrak{a}[(z,z_\Gamma),(y,y_\Gamma)] \qquad \text{ and } \qquad \mathfrak{a}[(y,y_\Gamma),(y,y_\Gamma)] \ge 0.$$
	We have $D(\Omega)\times H^{2}(\Gamma)\subset D(\mathfrak{a})$ is clearly dense in $\mathbb{L}^{2}$.
	 Now, let us show that $\mathfrak{a}$ is closed. It is obvious that 
	\begin{eqnarray}
		\|(y,y_\Gamma)\|^{2}_{D(\mathfrak{a})}&:=&\mathfrak{a}[(y,y_\Gamma),(y,y_\Gamma)]+\|(y,y_\Gamma)\|^{2}_{\mathbb{L}^{2}} \nonumber\\
		&\geq & d\|\Delta y\|^{2}_{L^{2}(\Omega)}	+\delta\|\Delta_{\Gamma} y_{\Gamma}\|^{2}_{L^{2}(\Gamma)}+\|(y,y_\Gamma)\|^{2}_{\mathbb{L}^{2}} \nonumber\\
		&\geq & (d\wedge \delta \wedge 1)\|(y,y_{\Gamma})\|^{2}_{\mathcal{H}^{2}}, \label{f1}
	\end{eqnarray}
    where $d\wedge \delta \wedge 1$ denotes the minimum of $d, \delta$ and $1$. The continuity of the trace and the normal derivative operators and estimate \eqref{f1} imply that $(D(\mathfrak{a}),\|\cdot\|_{D(\mathfrak{a})})$ is complete. This completes the proof.
\end{proof}

    We obtain first properties of the governing operator $\mathcal{A}$.
	\begin{theorem}\label{thm1} The following properties hold:
		\begin{enumerate}
			\item The operator $\mathcal{A}$ is self-adjoint and dissipative (nonpositive).
			\item The operator $\mathcal{A}$ generates a $C_0$-semigroup $\left(e^{t\mathcal{A}}\right)_{t\geq 0}$ on $\mathbb{L}^2$ that is self-adjoint, contractive and analytic of angle $\frac{\pi}{2}$.
		\end{enumerate} Moreover, the following interpolation identity holds
		$$[\mathbb L^2, D(\mathcal{A})]_{\frac{1}{2}}=D(\mathfrak{a}), $$
		where $[\cdot, \cdot]_{\frac{1}{2}}$ denotes the complex interpolation functor.
	\end{theorem}
	
	\begin{proof}
		Firstly, we prove that that $\mathcal{A}=-\widetilde{\mathcal{A}}$. Let $(y,y_\Gamma) \in D(\mathcal{A})$ and $(z,z_\Gamma)\in D(\mathfrak{a})$. Applying integration by parts and divergence formula \eqref{sdt} twice, we obtain
	\begin{align*}
			&\langle \mathcal{A}(y,y_\Gamma), (z,z_\Gamma)\rangle_{\mathbb{L}^2}  =-d\int_\Omega \Delta^2 y\, z \d x +d\int_\Gamma \partial_\nu (\Delta y) z_\Gamma \d S - \delta\int_\Gamma \Delta^2_\Gamma y_\Gamma\, z_\Gamma \d S\\
			&= d\int_\Omega \nabla\Delta y\, \nabla z \d x -d\int_\Gamma \partial_{\nu}(\Delta y)z\d S +d\int_\Gamma \partial_\nu (\Delta y) z_\Gamma \d S - \delta\int_\Gamma \Delta_\Gamma y_\Gamma\; \Delta_\Gamma z_\Gamma \d S\\
			&= -d\int_\Omega \Delta y\, \Delta z \d x +d\int_\Gamma \partial_{\nu}z \Delta y \d S-d\int_\Gamma \partial_{\nu}(\Delta y)z \d S +d\int_\Gamma \partial_\nu (\Delta y) z_\Gamma \d S \\
			&\quad - \delta\int_\Gamma \Delta_\Gamma y_\Gamma\; \Delta_\Gamma z_\Gamma \d S\\
			&= -d\int_\Omega \Delta y\, \Delta z \d x -d\int_\Gamma \partial_{\nu}(\Delta y)(z-z_{\Gamma}) \d S - \delta\int_\Gamma \Delta_\Gamma y_\Gamma\; \Delta_\Gamma z_\Gamma \d S \\
			&= -\mathfrak{a}[(y,y_\Gamma),(z,z_\Gamma)].
		\end{align*}
		Then, $D(\mathcal{A}) \subseteq D(\widetilde{\mathcal{A}})$ and $\mathcal{A}(y,y_\Gamma)=-\widetilde{\mathcal{A}}(y,y_\Gamma)$ for all $(y,y_\Gamma) \in D(\mathcal{A})$. Conversely, let
		$(y,y_{\Gamma})\in D(\widetilde{\mathcal{A}})$ and $(f,f_{\Gamma}):=\widetilde{\mathcal{A}}(y,y_{\Gamma})\in \mathbb{L}^{2}$. Then, for all $(z,z_{\Gamma})\in D(\mathfrak{a})$, we have
		\begin{align}
			& d\int_\Omega \Delta y \Delta z \,\d x +\frac{d}{\kappa}\int_{\Gamma}(y-  y_{\Gamma})(z- z_{\Gamma})\d S \nonumber\\
			& \quad + \delta\int_\Gamma \Delta_\Gamma y_\Gamma \Delta_\Gamma z_\Gamma \,\d S=\int_{\Omega} fz\d x + \int_{\Gamma} f_{\Gamma}z_{\Gamma}\d S.
			\label{w1}
		\end{align}
		Fixing $z=0$ in $\Omega$, for all $z_{\Gamma}\in 
		H^{2}(\Gamma)$, we deduce that
		\begin{equation*}
			\frac{d}{\kappa}\int_{\Gamma}(y-  y_{\Gamma})z_{\Gamma}\d S+ \delta\int_\Gamma \Delta_\Gamma y_\Gamma \Delta_\Gamma z_\Gamma \,\d S =\int_{\Gamma} f_{\Gamma}z_{\Gamma}\d S.
		\end{equation*}
  Since $z_\Gamma\in H^{2}(\Gamma)$ is arbitrary, this implies that $y_{\Gamma}$ is a weak solution of the elliptic equation
		\begin{equation}
			\delta \Delta_{\Gamma}^{2}y_{\Gamma}=f_{\Gamma}-\frac{d}{\kappa}(y-  y_{\Gamma})\in L^{2}(\Gamma). \label{form3}
		\end{equation}
		Recall
			that $\Gamma$ is a compact submanifold of class $C^2$ without boundary. Hence, we can apply the regularity given by the isomorphism \eqref{iso} to deduce that $y_{\Gamma}\in H^{4}(\Gamma)$.
		Choosing $z_{\Gamma}=0$ in \eqref{w1}. This yields
		\begin{equation}
			d\int_\Omega \Delta y \Delta z \,\d x +\frac{d}{\kappa}\int_{\Gamma}(y-  y_{\Gamma})z\d S=\int_{\Omega} fz\d x\qquad \forall z\in H^{2}(\Omega),
			\label{w2}
		\end{equation}
		with $\partial_{\nu}z=0$. In particular
  \begin{align*}
			&& d\int_\Omega \Delta y \Delta z \,\d x =\int_{\Omega} fz\d x\qquad \forall z\in D(\Omega),
		\end{align*}
  this implies that the distributional derivative $\Delta^{2}y$ belongs to $L^{2}(\Omega)$
and satisfies $$d\Delta^{2}y=f.$$
Returning to \eqref{w2}, we obtain
\begin{equation*}
-d\int_{\Gamma}\partial_{\nu}\Delta y z \d x+ \frac{d}{\kappa}\int_{\Gamma}(y-  y_{\Gamma})z\d S=0 \qquad \forall z\in H^{2}(\Omega)\;\text{with}\;\partial_{\nu}z=0.
\end{equation*}
Note that for any $w\in H^{3/2}(\Gamma)$, we can find 
$z\in H^{2}(\Omega)$ such that $z=w$ and $\partial_{\nu}z=0$ on $\Gamma$, thus, 
$$-d\int_{\Gamma}\partial_{\nu}\Delta y w \d x+ \frac{d}{\kappa}\int_{\Gamma}(y-  y_{\Gamma})w\d S=0 \qquad \forall w\in H^{3/2}(\Gamma).
$$
By density argument, it follows that
  \begin{align*}
      \partial_{\nu}(\Delta y)=\frac{1}{\kappa}(y- y_{\Gamma}).
  \end{align*}
  Consequently, $y$ is a weak solution to the following problem
		\begin{equation}
			\left\{
			\begin{aligned}
				& d\Delta^{2} y=f\in L^{2}(\Omega),
				& & \text {in}\;\Omega, \\
				& \partial_{\nu}(\Delta y)=\frac{1}{\kappa}(y- y_{\Gamma})\in H^{1/2}(\Gamma), & & \text {on}\;\Gamma.
			\end{aligned}
			\right. \label{form4}
		\end{equation}
		This allows us to apply elliptic regularity theory for Poisson’s equation with inhomogeneous Neumann boundary condition (see \cite[Proposition 7.5]{Ta'11}) to deduce that 
		$\Delta y\in
			H^{2}(\Omega)$
		which is combined with $\partial_{\nu}y=0$ to conclude that $y\in H^{4}(\Omega)$.\\
		Consequently $(y,y_{\Gamma})\in D(\mathcal{A}),$ and from \eqref{form3} and \eqref{form4}, we obtain $\mathcal{A}(y,y_{\Gamma})=-\widetilde{\mathcal{A}}(y,y_{\Gamma})$.
        
		Since the form $\mathfrak{a}$ is symmetric and accretive, the operator $\mathcal{A}$ is self-adjoint and dissipative. It follows from \cite[Theorem 1.52]{Ou'05} that $\mathcal{A}$ generates a $C_0$-semigroup on $\mathbb{L}^2$ that is self-adjoint, contractive, and analytic. Therefore, the operator $\mathcal{A}$ is self-adjoint and bounded above (it is dissipative). Then, \cite[Corollary 4.7]{EN'00} yields that $\mathcal{A}$ generates an analytic $C_0$-semigroup of angle $\frac{\pi}{2}$.
		
		Now, let $\lambda>0$. By \cite[Theorem VI.2.23]{Ka'95}, we have $$D\left((-\mathcal{A})^\frac{1}{2}\right)=D\left((\lambda-\mathcal{A})^\frac{1}{2}\right)=D(\mathfrak{a}).$$
		On the other hand, it follows from \cite[Theorem 2.33]{Ya'10} that $$[\mathbb L^2, D(\mathcal{A})]_{\frac{1}{2}}=D\left((-\mathcal{A})^\frac{1}{2}\right).$$
		Thus, the proof is completed.
	\end{proof}
	
	\begin{remark}\label{rmk1}
		The operator $\mathcal{A}$ is nonpositive, that is
		$$\forall (u,u_\Gamma)\in D(\mathcal{A}), \qquad \langle \mathcal{A}(u,u_\Gamma), (u,u_\Gamma)\rangle_{\mathbb L^2} \le 0.$$
		However, it is not negative. Indeed, the null space
		$$\mathcal{N}(\mathcal{A}):=\left\{(u,u_\Gamma)\in D(\mathcal{A}) \colon \mathcal{A}(u,u_\Gamma)=0 \right\}$$
		contains the set of constant functions 
		$$\mathcal{C}:=\left\{\left(c,c\right) \colon c \text{ is a constant function in } \overline{\Omega}\right\}.$$
	\end{remark}
	
	\begin{lemma}
		The operator $\mathcal{A}$ has compact resolvent, and the $C_0$-semigroup $\left(\mathrm{e}^{t\mathcal{A}}\right)_{t\geq 0}$ is immediately compact, i.e., the operator $\mathrm{e}^{t\mathcal{A}}$ is compact on $\mathbb{L}^2$ for all $t>0$.
	\end{lemma}
	\begin{proof}
		To show that $\mathcal{A}$ has compact resolvent, it suffices to prove that the injection of $(D(\mathfrak{a}), \|\cdot\|_{\mathfrak{a}})$ into $\mathbb{L}^2$ is compact. Using \eqref{f1}, we obtain $(D(\mathfrak{a}), \|\cdot\|_{\mathfrak{a}})\hookrightarrow(\mathcal{H}^2, \|\cdot\|_{\mathcal{H}^2})$ is bounded, the embeddings $H^{2}(\Omega) \hookrightarrow L^2(\Omega)$ and $H^{2}(\Gamma) \hookrightarrow L^2(\Gamma)$
		are both compact by the Sobolev embedding theorem. Then $(D(\mathfrak{a}), \|\cdot\|_{\mathfrak{a}}) \hookrightarrow \mathbb{L}^2$ is compact.
		Since $\left(\mathrm{e}^{t\mathcal{A}}\right)_{t\geq 0}$ is immediately norm continuous (it is analytic), then \cite[Theorem 4.29]{EN'00} implies that $\left(\mathrm{e}^{t\mathcal{A}}\right)_{t\geq 0}$ is immediately compact.
	\end{proof}
	Therefore, we have the following spectral result:
	\begin{corollary}\label{c.spectrum}
		There exists an orthonormal basis $(\Phi_{n})_{n\in\mathbb{N}}$ of $\mathbb{L}^2$ consisting of eigenfunctions of $\mathcal{A}$, i.e.,
		$\mathcal{A} \Phi_{n} = -\lambda_{n} \Phi_{n}$, where the sequence $(\lambda_{n})_{n\in \mathbb N}$ is such that $0=\lambda_{1} <\lambda_{2} \leq\dots \lambda_{n} \leq \ldots\to \infty$. In particular, $s(-\mathcal{A})=0$ (the spectral bound). Moreover, the $C_0$-semigroup is given by
		\begin{equation} \label{sgdecomp}
			\mathrm{e}^{t\mathcal{A}} \mathfrak{u}=\sum_{n=1}^\infty \mathrm{e}^{-\lambda_{n} t}\langle \mathfrak{u},\Phi_{n}\rangle_{\mathbb{L}^2}\Phi_{n} \qquad \forall \mathfrak{u} \in \mathbb{L}^2.
		\end{equation}
		Furthermore, the null space of $\mathcal{A}$ is given by $\mathcal{N}(\mathcal{A})=\mathcal{C}$.
	\end{corollary}
	
	\begin{proof}
		Since $\mathcal{A}$ is self-adjoint and has compact resolvent, most of the above facts follow from the standard spectral theory. By Remark \ref{rmk1}, we have $\mathcal{C} \subset \mathcal{N}(\mathcal{A})$. Then the first eigenvalue of $\mathcal{A}$ is $-\lambda_{1}=0$. Conversely, for $(y,y_\Gamma) \in \mathcal{N}(\mathcal{A})$, we have
		\begin{align*}
			\mathfrak a[(y,y_\Gamma),(y,y_\Gamma)]&= d\int_\Omega |\Delta y|^2 \,\d x +\frac{d}{\kappa}\int_{\Gamma}(y- y_{\Gamma})^{2}\d S+ \delta\int_\Gamma |\Delta_\Gamma y_\Gamma|^2 \,\d S \\
			&= \langle -\mathcal{A}(y,y_\Gamma),(y,y_\Gamma)\rangle_{\mathbb L^2}=0.
		\end{align*}
		Then $\Delta y=0$ in $\Omega$ and $y- y_{\Gamma}=0$ on $\Gamma$. Multiplying the first equation by $y$ and using Green's formula in $\Omega$ with $\partial_{\nu}y=0$, we obtain $\nabla y=0$. Since $\Omega$ is connected, it follows that $y=c$ is constant on $\Omega$. From $y-y_{\Gamma}=0$ in $\Gamma$, we obtain $y_{\Gamma}=c$ in $\Gamma$. This completes the proof.
	\end{proof}
	Consequently, we have the following asymptotic behavior of the semigroup $\left(\mathrm{e}^{t\mathcal{A}}\right)_{t\geq 0}$, which follows immediately by Parseval's identity.
	\begin{corollary}
		The $C_0$-semigroup $\left(\mathrm{e}^{t\mathcal{A}}\right)_{t\geq 0}$ satisfies
		$$\|\mathrm{e}^{t\mathcal{A}}\|_{\mathcal{L}(\mathbb{L}^2,\mathbb{L}^2)}\leq 1 \qquad \mbox{and} \qquad \|\mathrm{e}^{t\mathcal{A}}-\mathds{P}\|_{\mathcal{L}(\mathbb{L}^2,\mathbb{L}^2)}\leq e^{-\lambda_{2} t} \qquad \forall t\ge 0,$$
		where $\mathds{P}$ is the orthogonal projection of $\,\mathbb{L}^{2}$ onto the vector space spanned by $\mathds{1}$: 
		$$\mathds{P}\mathfrak{u}:=\frac{1}{\mu(\overline{\Omega})}\left(\int_{\Omega}u\d x+ \int_{\Gamma}u_{\Gamma}\d S\right)\mathds{1}\qquad \forall \mathfrak{u}=(u,u_{\Gamma})\in \mathbb{L}^2.$$
	\end{corollary}
	In the sequel, we will use the following notions of solutions.
	\begin{definition}
		Let $(f,f_\Gamma) \in L^{2}(0,T;\mathbb{L}^{2})$ and $Y_0:=(y_0, y_{0,\Gamma})\in \mathbb{L}^2$.
		\begin{enumerate}[label=(\alph*),leftmargin=*]
			\item A strong solution of \eqref{eq1to4} is a function $Y:=(y, y_{\Gamma}) \in \mathbb{Y}_1 :=H^1(0,T;\mathbb{L}^2) \cap L^2(0,T;D(\mathcal{A}))$ fulfilling \eqref{eq1to4} in $L^2(0,T; \mathbb{L}^2)$.
			\item A mild solution of \eqref{eq1to4} is a function $Y:=(y, y_{\Gamma}) \in C([0,T]; \mathbb{L}^2)$ satisfying, for $t\in [0,T]$,
			$$Y(t,\cdot)=\mathrm{e}^{t\mathcal{A}} Y_0 + \int_0^t e^{(t-\tau)\mathcal{A}} [f(\tau, \cdot), f_\Gamma(\tau, \cdot)] \d\tau.$$
		\end{enumerate}
	\end{definition}
	
	The analyticity of the semigroup $\left(\mathrm{e}^{t\mathcal{A}}\right)_{t\geq 0}$ generated by the system \eqref{eq1to4} implies the following regularity result (see, e.g., \cite[Theorem 3.4]{Ya'10}).
	\begin{proposition}
		Let $(f,f_\Gamma) \in L^{2}(0,T;\mathbb{L}^{2})$.
		\begin{enumerate}[label=(\roman*),leftmargin=*]
			\item  For all $Y_0:=(y_0,y_{0,\Gamma})\in D(\mathfrak{a})$, there exists a unique strong solution of \eqref{eq1to4}.
            Moreover, there exists a constant $C>0$ (independent of $Y_0$ and $(f,f_\Gamma)$) such that
            \begin{eqnarray}
                \|Y\|_{\mathbb{Y}_1}\leq C\left(\|Y_{0}\|_{D(\mathfrak{a})}+\|(f,f_{\Gamma})\|_{L^{2}(0,T;\mathbb{L}^{2})} \right). 
            \end{eqnarray}
			\item For all $Y_0:=(y_0,y_{0,\Gamma})\in \mathbb{L}^2$, there exists a unique mild solution of \eqref{eq1to4} $Y:=(y,y_{\Gamma})\in C([0,T]; \mathbb{L}^2)$.
            Moreover, there exists a constant $C>0$ (independent of $Y_0$ and $(f,f_\Gamma)$) such that
            \begin{eqnarray}
                \|Y\|_{C([0,T]; \mathbb{L}^2)}\leq C\left(\|Y_{0}\|_{\mathbb{L}^{2}}+\|(f,f_{\Gamma})\|_{L^{2}(0,T;\mathbb{L}^{2})} \right). \label{energy estimate}
            \end{eqnarray}
            
		\label{prop2}
        \end{enumerate}
	\end{proposition}
	
    \begin{remark}
        In the absence of surface diffusion, i.e., $\delta=0$, certain properties, such as $H^{4}$
  regularity in the bulk and compactness, are no longer guaranteed. However, it is still possible to prove well-posedness and regularity results adapted to the following domains: 
  \begin{align*}
      D(\mathfrak{a})&=\left\{(y,y_{\Gamma})\in H^{2}(\Omega)\times L^{2}(\Gamma)\;:\; \partial_{\nu}y=0\right\},\\
      D(\mathcal{A})&=\left\{(y,y_{\Gamma})\in H^{2}(\Omega)\times L^{2}(\Gamma)\;:\; \Delta^{2}y\in L^{2}(\Omega),\; \partial_{\nu}y=0,\; 
 \kappa\partial_{\nu}\Delta y=y-y_{\Gamma} \right\}.
  \end{align*}
    \end{remark}
    
	\section{Eventual positivity of the semigroup} \label{sec4}
	In this section, we analyze some positivity properties related to the solution of \eqref{eq1to4}. We recall that the space $\mathbb{L}^2$ is a Hilbert lattice whose positive cone is the product of the positive cones of $L^2(\Omega)$ and $L^2(\Gamma)$, respectively. For any any real functions $u$ and $v$ defined in $\Omega$ or on $\Gamma$, respectively, and $\mathfrak{u}:=(u,u_\Gamma)\in \mathbb{L}^2 $, we adopt the following notations:
 \begin{align*}
     &&u \vee v:=\max(u,v),   \qquad v\wedge u:=\min(u,v),\\
     &&u^{+}:=u\vee 0,  \qquad \qquad\quad\quad u^{-} :=-(u\wedge 0),\\
     &&\mathfrak{u}^+:=(u^+, u_\Gamma^+), \qquad \qquad\quad \mathfrak{u}^- :=(u^-, u_\Gamma^-).
 \end{align*}

 \begin{proposition}\label{nonpos} 
		The following properties hold:
		\begin{enumerate}[label=(\roman*),leftmargin=*]
			\item The $C_0$-semigroup $\left(\mathrm{e}^{t\mathcal{A}}\right)_{t\geq 0}$ is not positive.
			\item The $C_0$-semigroup $\left(\mathrm{e}^{t\mathcal{A}}\right)_{t\geq 0}$ is not $\mathbb{L}^{\infty}$-contractive.
		\end{enumerate}
	\end{proposition}
	\begin{proof}
		$(i)$ A necessary condition for $\left(\mathrm{e}^{t\mathcal{A}}\right)_{t\geq 0}$ to be positive is that $\mathfrak{u}^+ \in D(\mathfrak{a})$ whenever $\mathfrak{u} \in D(\mathfrak{a})$ is a real-valued function; see \cite[Theorem 2.6]{Ou'05}. However, this is not the case. 
		Indeed, we fix $y:=(y_{1},\cdots,y_{N})$ in $\Omega$ and $\varepsilon>0$ such that $\overline{B}(y,2\varepsilon)\subset \Omega$. 
		Let $\vartheta\in D(B(y,2\varepsilon))$ such that $\vartheta=1$ on $B(y,\varepsilon)$ ($B(y,r)$ and $\overline{B}(y,r)$ are respectively the open ball and the closed ball of center $y$ and radius $r$). We set
		$$\psi(x):=(x_{1}-y_{1})\vartheta(x)\qquad \forall x=(x_{1},\cdots, x_{N})\in\Omega.$$
		Clearly, we have $\Psi:=(\psi,0)\in D(\mathfrak{a})$. Moreover, $\Psi^{+}:=(\psi^{+},0)\notin D(\mathfrak{a})$. Indeed,
		$$\psi(x):=x_{1}-y_{1} \qquad \forall x\in B(y,\varepsilon).$$
		Consequently, on $B(y,\varepsilon)$, we obtain
		$$\psi^{+}(x)=(x_{1}-y_{1})\mathds{1}_{B(y,\varepsilon)\cap \{x_{1}>y_{1}\}}$$
		and therefore, 
		$$\partial_{x_{1}}\psi^{+}|_{B(y,\varepsilon)}=\mathds{1}_{B(y,\varepsilon)\cap \{x_{1}>y_{1}\}}\notin H^{1}(B(y,\varepsilon)),$$
		which implies that $\psi^{+}\notin H^{2}(\Omega)$.
		
		$(ii)$ A necessary condition for $\left(\mathrm{e}^{t\mathcal{A}}\right)_{t\geq 0}$ to be $\mathbb{L}^{\infty}$-contractive is that $\mathfrak{u}\wedge \mathds{1} \in D(\mathfrak{a})$ whenever $\mathfrak{u} \in D(\mathfrak{a})$ is a positive valued function; see \cite[Theorem 2.13]{Ou'05}. We can also assume that $\varepsilon\in (0,\frac{1}{2})$ and $0\leq\vartheta\leq 1$. We put $\Phi=(\varphi, 1)$ where
		$$\varphi(x):=1-\psi(x)\qquad \forall x=(x_{1},\cdots, x_{N})\in\Omega.$$
		Note that $\Phi\in D(\mathfrak{a})$ is positive. On $B(y,\varepsilon)$, we obtain 
		$$(\varphi\wedge 1)(x)=1-(x_{1}-y_{1})\mathds{1}_{B(y,\varepsilon)\cap \{y_{1}<x_{1}<y_{1}+\varepsilon\}},$$
		and therefore, 
		$$\partial_{x_{1}}(\varphi\wedge 1)|_{B(y,\varepsilon)}=-\mathds{1}_{B(y,\varepsilon)\cap \{y_{1}<x_{1}<y_{1}+\varepsilon\}}\notin H^{1}(B(y,\varepsilon)),$$
		which implies that $\varphi\wedge 1\notin H^{2}(\Omega)$.
	\end{proof}
	
	We need the following lemma to prove the uniform eventual Markovianity of the semigroup $\left(\mathrm{e}^{t\mathcal{A}}\right)_{t\geq 0}$.
	\begin{lemma}\label{lmdom}
		For all $p\in\mathbb{N}$, we have 
		$$D(\mathcal{A}^{p})\subset \mathcal{H}^{2(p+1)}.$$
		In particular, $D(\mathcal{A}^{\infty})\subset \mathbb{L}^{\infty}.$
	\end{lemma}
	\begin{proof}
		We argue by induction. 
		For $p=1$, the property is trivial. Let $\mathfrak{u}:=(u,u_\Gamma)\in D(\mathcal{A}^{p+1})$. Then $\mathfrak{u}\in D(\mathcal{A}^{p})$ and $(f,f_{\Gamma}):=\mathcal{A}\mathfrak{u}\in D(\mathcal{A}^{p})$. In one hand, we see that $u_{\Gamma}$ solves
		\begin{align*}
			(-\Delta_{\Gamma}+1)^{2}u_{\Gamma}= \delta^{-1}(u-u_{\Gamma}-f_{\Gamma})-2\Delta_{\Gamma}u_{\Gamma}+u_{\Gamma}\in H^{2p}(\Gamma).
		\end{align*}
		Using \eqref{iso}, 
		we obtain
		$u_{\Gamma}\in H^{2(p+2)}(\Gamma)$. On the other hand, $u$ solves the following Poisson’s equation with inhomogeneous
		Neumann boundary condition
		\begin{equation}
			\left\{
			\begin{aligned}
				& \Delta^{2} u +\Delta u=F:=-f+ \Delta u\in H^{2p}(\Omega),
				& & \text {in}\;\Omega, \\
				& \partial_{\nu}(\Delta u)=G:=(u- u_{\Gamma})\in
				H^{2p+1/2}(\Gamma),  & & \text {on}\;\Gamma.
			\end{aligned}
			\right. 
		\end{equation}
		This allows us to apply the elliptic regularity theory (see \cite[Proposition 7.5]{Ta'11}) to deduce that $\Delta u\in H^{2p+2}(\Omega),$
		which is combined with $\partial_{\nu}u=0$ to conclude that $u\in H^{2p+4}(\Omega)$.
        
		For $p$ large enough, the Sobolev embedding entails that $D(\mathcal{A}^{\infty})\subset D(\mathcal{A}^{p})\subset \mathbb{L}^{\infty}.$ See, e.g., \cite[Theorem 4.12]{AF'03} and \cite[Theorem 2.2]{LP'87}.
	\end{proof}
		\begin{remark}
		Since the eigenfunctions $\{\Phi_{n}\}$ of $\mathcal{A}$ belong to $D(\mathcal{A}^{\infty})$, they admit representatives of class $C^{\infty}(\overline{\Omega})$.
	\end{remark}
    As final consequence, we obtain the following Markovianity result.
	\begin{theorem}
		The $C_0$-semigroup $\left(\mathrm{e}^{t\mathcal{A}}\right)_{t\geq 0}$ is uniformly eventually Markovian.
	\end{theorem}
	\begin{proof}
		By Lemma \ref{lmdom} and Proposition \ref{subM}, it follows that $\left(\mathrm{e}^{t\mathcal{A}}\right)_{t\geq 0}$ is uniformly eventually sub-Markovian. To show that $\mathrm{e}^{t\mathcal{A}} \mathds{1}=\mathds{1}$, it suffices to observe that $\mathds{1}\in D(\mathcal{A})$, $\mathcal{A} \mathds{1}=\mathbf{0}$, and use the identity
		$$\mathrm{e}^{t\mathcal{A}} \mathds{1}- \mathds{1}=\int_0^t \mathrm{e}^{\tau\mathcal{A}} \mathcal{A}\mathds{1}\, \d \tau.$$
	\end{proof}

\section{Conclusion}\label{sec7}
In this work, we have studied of the biharmonic heat equation with general dynamic boundary conditions that involve the bi-Laplace-Beltrami operator and a Robin-type bulk-surface coupling. We started with the well-posedness and regularity of the solution associated with our system, by combining semigroup and form methods. For this purpose, we introduced a suitable functional framework that allowed us to prove some qualitative properties of the solutions. In particular, we have proven the generation of an analytic $C_0$-semigroup on the product space $L^2(\Omega)\times L^2(\Gamma)$. Furthermore, we have proven the failure of classical positivity and $L^\infty$-contractivity of the semigroup, whereas we have established its eventual positivity and eventual $L^\infty$-contractivity.

Our results pave the way for the study of further related problems, such as controllability and inverse problems for fourth-order parabolic equations with dynamic boundary conditions. Moreover, the numerical analysis of such systems is non-classical and would be of particular interest for applications. These directions deserve further investigation in view of their significance in the modeling of various applied phenomena.

\end{document}